\documentclass[11pt]{amsart} \textwidth=14.5cm \oddsidemargin=1cm
\usepackage{amsmath, amstext, amsbsy, amssymb, amscd, mathrsfs}
\usepackage{amsmath}
\usepackage{amsxtra}
\usepackage{amscd}
\usepackage{amsthm}
\usepackage{amsfonts}
\usepackage{amssymb}
\usepackage{eucal}
\usepackage{color}
\usepackage[all]{xy}
\usepackage{bbding}
\usepackage{pmgraph}
\usepackage{stmaryrd}
\usepackage{graphicx}
\usepackage{bm}
\usepackage[enableskew,vcentermath]{youngtab}

\newtheorem{theorem}{Theorem}[section]

\newtheorem{lemma}[theorem]{Lemma}
\newtheorem{proposition}[theorem]{Proposition}
\newtheorem{corollary}[theorem]{Corollary}
\newtheorem{example}[theorem]{Example}
\theoremstyle{definition}
\newtheorem{definition}[theorem]{Definition}

\theoremstyle{remark}
\newtheorem{remark}[theorem]{Remark}

\definecolor{A}{rgb}{.75,1,.75}

\numberwithin{equation}{section}

\numberwithin{equation}{section}

\newcommand{\ds}{\displaystyle}

\newcommand{\C}{\mathbb C}
\newcommand{\Z}{\mathbb Z}

\newcommand{\mf}{\mathfrak}

\newcommand{\Hom}{\text{Hom} }
\newcommand{\undm}{\underline{m}}
\newcommand{\al}{\alpha}

\newcommand{\ee}{\varepsilon}

\newcommand{\mc}{\mathcal}
\newcommand{\HC}{\mathcal{H}_n}
\newcommand{\la}{\lambda}

\newcommand{\Ga}{\Gamma}
\newcommand{\La}{\Lambda}

{\vskip-\lastskip\medskip
  \noindent
  {\em #1.}\enspace
  }%
{\qed\par\medskip
  }

\begin{document}

\title[Centers of Hecke algebras ]{Frobenius map for the centers of Hecke algebras}
\author[Jinkui Wan and Weiqiang Wang]{Jinkui Wan and Weiqiang Wang}
\address{
(Wan) School of Mathematics, Beijing Institute of Technology,
Beijing, 100081, P.R. China. } \email{wjk302@gmail.com}

\address{(Wang) Department of Mathematics, University of Virginia,
Charlottesville,VA 22904, USA.}
\email{ww9c@virginia.edu}

\begin{abstract}
We introduce a commutative associative graded algebra structure on
the direct sum $\mathcal Z$ of the centers of the Hecke algebras associated
to the symmetric groups in $n$ letters for all $n$.
As a natural deformation of the classical construction
 of Frobenius, we
establish an algebra isomorphism from $\mathcal Z$ to the ring
of symmetric functions. This isomorphism provides
an  identification between several distinguished
bases for the centers (introduced by Geck-Rouquier, Jones, Lascoux)
 and explicit bases of symmetric functions.
\end{abstract}


\maketitle

\section{Introduction}

Frobenius~\cite{Fro} constructed a natural isomorphism $\psi$ from
the direct sum of the centers of the symmetric group algebras
to  the space of symmetric functions, which
sends the conjugacy class sums to the power-sum symmetric functions
(up to some scalars).
This gives rise to the characteristic map
under which the irreducible characters  of symmetric groups
are matched with the Schur functions (cf. \cite[I. 7]{Mac}).
This construction of Frobenius has played a fundamental
role in the interaction between  the representation
theory of symmetric groups and combinatorics of symmetric functions.

The Hecke algebra $\HC$
is a $v$-deformation of the group algebra of the symmetric group $\mf S_n$.
The structures of the centers $\mathcal Z(\HC)$
of the Hecke algebras have been the subject of
studies by various authors (cf. \cite{DJ, FrG, FW, GR, Jo, La}).
Using the class polynomials for Hecke algebras \cite{GP},
Geck and Rouquier \cite{GR} constructed a basis of $v$-class sums
for the center $\mathcal Z(\HC)$,
which specializes at $v=1$ to the usual class sums of $\mf S_n$. There is a neat
characterization of these $v$-class sums by Francis \cite{Fra} though
no closed formula is known.
Extending a basis of Jones \cite{Jo} of the center $\mc Z(\HC)$,
Lascoux~\cite{La} constructed several more bases of $\mathcal Z(\HC)$, and moreover,
he  provided an elegant description of the transition matrices from these bases
to the Geck-Rouquier basis in terms of the transition
matrices of well-known bases  in the ring of symmetric functions.
Lascoux stated the identifications
of these bases of $\mc Z(\HC)$ with  symmetric functions
 as an open problem.

The goal of this paper is to formulate a Frobenius map
$\Psi: \mathcal Z \longrightarrow \Lambda_v$,
 which is a $v$-deformation of the classical Frobenius map $\psi$,
from  the direct sum
$\mathcal Z$ of the centers of $\HC$ for all $n$
to the $\C(v)$-space $\Lambda_v$ of symmetric functions.
We establish various basic properties of $\Psi$ including a precise connection to
the character formula in \cite{Ram}, and in particular,
we provide a positive answer to Lascoux's question of identifying
bases in $\mathcal Z$ and $\Lambda_v$.
It is our belief that the Frobenius map $\Psi$ will play a
fundamental role in
further understanding of the centers of Hecke algebras with
powerful tools from the theory of symmetric functions.

Let us explain in some detail.
As in the case of symmetric groups, considering the centers
of $\HC$ for all $n$ simultaneously is again crucial in the formulation of this paper
(surprisingly, this view was not taken advantage of in the literature).
We  first construct a graded linear isomorphism  $\Psi: \mathcal Z \longrightarrow \Lambda_v$
by sending the Geck-Rouquier class elements
to a variant of monomial symmetric functions in the $\la$-ring notation.
Our Frobenius map $\Psi$ is compatible with Lascoux's earlier remarkable
work on transition matrices between bases, and in this way,
we are able to identify precisely (see Theorem~\ref{PsiLascoux})
various bases of $\mc Z(\HC)$ in \cite{La}
with well-known bases in $\Lambda_v$ of degree $n$
up to some $\la$-ring notation shifts, that is, $\{s_\la\}, \{m_\la\},
\{e_\la\}, \{h_\la\}$ and $\{p_\la\}$ in
Macdonald's notation.

Via a notion of relative norm on Hecke algebras
(which goes back to unpublished work of Peter Hoefsmit and Leonard Scott in 1976
according to \cite{Jo}), we show that
$\mathcal Z$ is naturally a commutative associative graded algebra and that
$\mathcal Z$ is isomorphic to a polynomial algebra. Then we show that
$\Psi: \mathcal Z \rightarrow \Lambda_v$
is actually an isomorphism of graded algebras. Our proof
uses the identification under $\Psi$ of one basis of Lascoux for $\mathcal Z$
with a multiplicative basis for $\Lambda_v$.
Using $\Psi$, we are able to provide new simple
proofs for Lascoux's difficult results on some other bases of $\mathcal Z$.

The Frobenius map $\Psi$ is shown to fit perfectly with the character formula
for the Hecke algebras established in~\cite{Ram}.
Moreover, $\Psi$ readily gives rise to a $v$-characteristic map $\text{ch}_v$
for the Hecke algebra $\HC$
which sends the irreducible characters to the corresponding Schur functions; see
the commutative diagram \eqref{eq:cd} relating $\Psi$ to $\text{ch}_v$.
In our approach, $\Psi$ plays a more fundamental role than $\text{ch}_v$ just
as in the original formulation of Frobenius.

As pointed out by the referee, in light of the commutative diagram \eqref{eq:cd}, in principle one could present the results of this paper by turning
around the constructions of Frobenius map $\Psi$ and the $v$-characteristic map ${\rm ch}_v$ (we were aware of this option too).
More precisely, one can first define
the $v$-characteristic map $\text{ch}_v$ by sending the irreducible characters of Hecke algebras to the corresponding Schur functions.
Via the commutative diagram \eqref{eq:cd}, $\text{ch}_v$ gives rise to
a {\em linear isomorphism} $\Psi:\mc Z\rightarrow\La_v$, which
sends primitive central idempotents to Schur functions (up to scalars).
Dualizing Frobenius character formula in \cite{Ram},
one readily shows that $\Psi$ maps the Geck-Rouquier class elements
to a variant of monomial symmetric functions in the $\la$-ring notation (which we took as the definition of $\Psi$ in this paper).
However, we do not know of a direct proof that $\Theta$ in \eqref{eq:cd} is an {\em algebra} isomorphism,
and so we can not derive that $\Psi$ is an algebra isomorphism except
by relating to Lascoux's results in \cite{La} as presented in this paper.
The way we choose to organize this paper has the additional advantage of making a large portion of the paper
on basic properties of $\Psi$ independent of Ram's character formula (which was highly nontrivial to establish).

Here is a layout of the paper.
In Section~\ref{sec:center}, we review some basic constructions
on Hecke algebras including Geck-Rouquier basis. In Section~\ref{sec:Psi},
we show that $\mathcal Z$ endowed with a product  from
relative norm is commutative, associative, and isomorphic to
a polynomial algebra.
The Frobenius map $\Psi$
is defined and shown to be an isomorphism of graded algebras.
The images of various bases of the centers
under the isomorphism $\Psi$ are determined in Section~\ref{sec:identification}.
In Section~\ref{sec:v-ch}, we show how $\Psi$ gives rise to a $v$-characteristic
map which identifies the irreducible characters with Schur functions.

{\bf Acknowledgements.}
We thank Shun-Jen Cheng and Institute of Mathematics, Academia Sinica, Taipei for support
and providing an excellent atmosphere in December 2011, during which
this project was initiated.
The first author was partially supported by
NSFC-11101031, and the second author was partially
supported by NSF DMS-1101268.

\section{Centers of Hecke algebras}
 \label{sec:center}

\subsection{The Hecke algebra $\HC$}
Let $\mf S_n$ be the symmetric group in $n$ letters generated
by the simple transpositions $s_i=(i,i+1), i=1,\ldots, n-1$.
Let $v$ be an indeterminate. The Hecke algebra $\HC$ is the unital
$\C(v)$-algebra generated by $T_1,\ldots, T_{n-1}$
satisfying the relations
\begin{align*}
(T_k-v)(T_k+1)&=0,\quad 1\leq k\leq n-1,\\
T_kT_{k+1}T_k&=T_{k+1}T_kT_{k+1},\quad 1\leq k\leq n-2,\\
T_kT_{\ell}&=T_{\ell}T_k,\quad |k-\ell|>1.
\end{align*}
If $w=s_{i_1}\cdots s_{i_r}$ is a reduced expression, then we define
$T_w=T_{i_1}\cdots T_{i_r}$.
The Hecke algebra $\HC$ is a free $\C(v)$-module with
basis $\{T_w~|w\in\mf S_n\}$, and it can be regarded as a $v$-deformation
of the group algebra $\C\mf S_n$ of the symmetric group $\mf S_n$.

Denote by $\mc P_n$ (resp. $\mc {CP}_n$) the set of partitions (resp. compositions) of $n$.
For each $\la\in\mc P_n$, there exists an irreducible character
$\chi^{\la}_v$ of $\HC$ such that $\{\chi^{\la}_v~|\la\in\mc P_n\}$ is a complete set
of irreducible characters of $\HC$; moreover the specialization of $\chi^{\la}_v$
at $v=1$ coincides with the character $\chi^{\la}$ of the Specht module of $\mf S_n$ associated to $\la$
for $\la\in\mc P_n$.

\subsection{Geck-Rouquier class elements}

For $\la\in\mc P_n$, denote by $C_{\la}$ the conjugacy class of $\mf S_n$
containing all permutations of cycle type $\la$, and accordingly let
$c_{\la}$ be the class sum of $C_{\la}$.
Clearly, the set $\{c_{\la}~|\la\in\mc P_n\}$ forms a linear basis of the center
of the group algebra $\C \mf S_n$.

There exists a canonical symmetrizing trace form $\tau$ on $\HC$ defined as follows~\cite{GP}
\begin{align*}
\tau:\HC \longrightarrow\C(v), \qquad
T_w \mapsto
\left\{
\begin{array}{ll}
1,&\text{ if }w=1,\\
0,&\text{ otherwise}.
\end{array}
\right.
\end{align*}
Then the bilinear form $(\cdot,\cdot)$ induced by $\tau$ on $\HC$ satisfies
$(T_w,T_\sigma)=\delta_{w,\sigma^{-1}}v^{\ell(w)}$ and hence the basis dual to
$\{T_w~|w\in\mf S_n\}$ is
$$
T_w^\vee :=v^{-\ell(w)}T_{w^{-1}}, \quad \text{ for } w\in\mf S_n.
$$
For $\chi\in\Hom_{\C(v)}(\HC,\C(v))$, define
\begin{equation}\label{dual}
\chi^*=\sum_{w\in\mf S_n}\chi(T_w)T_w^\vee.
\end{equation}
Then it is easy to check that $\chi$ is a trace function on $\HC$, that is, $\chi(hh')=\chi(h'h)$ for
$h,h'\in\HC$, if and only if $\chi^*$ is central in $\HC$.

For $\la=(\la_1,\la_2,\ldots,\la_{\ell})\in\mc P_n$,
set
$$
w_{\la}=(1,\ldots,\la_1)(\la_1+1,\ldots, \la_1+\la_2)
\cdots(\la_1+\cdots+\la_{\ell-1}+1,\ldots, \la_1+\cdots+\la_{\ell}).
$$
Observe that $w_{\la}$ is a minimal length element in
the conjugacy class $C_{\la}$.
It follows from \cite[Theorem~5.1]{Ram} (cf. \cite[Section~8.2]{GP}) that
for $w\in\mf S_n$, there exists $f_{w,\la}\in\Z[v,v^{-1}]$ such that
$$
T_w\equiv \sum_{\la\in\mc P_n}f_{w,\la}T_{w_{\la}}\mod [\HC,\HC],
$$
and moreover
\begin{align*}
f_\la:\HC \longrightarrow\C(v), \qquad
T_w \mapsto f_{w,\la}
\end{align*}
is a trace function.
Then we obtain the Geck-Rouquier central elements
$$
f_\la^*=\sum_{w\in\mf S_n}f_{w,\la}T_w^\vee
=\sum_{w\in\mf S_n}v^{-\ell(w)}f_{w,\la}T_{w^{-1}}, \quad \text{ for }\la\in\mc P_n.
$$
Observe that $f^*_{\la}$ specializes at $v=1$ to the class sum $c_{\la}$.

\begin{lemma}\cite{GR}
The set $\{f_\la^*~|\la\in\mc P_n\}$ forms a basis of the center
$\mc Z(\HC)$ of $\HC$.
\end{lemma}


The definition of the Hecke algebra $\HC$ used in this paper is
different from the one used in \cite{La}, where the parameter $q$ is used.
The generators in \cite{La}, which we denote by  $\widetilde{T}_1,\ldots,
\widetilde{T}_{n-1}$, are related to $T_1,\ldots, T_{n-1}$ via
$v=q^2$ and $\widetilde{T}_i=v^{-\frac 12}T_i$, for $1\leq i\leq n-1$.
For $w\in \mf S_n$, we set
\begin{equation} \label{eq:Ttilde}
\widetilde{T}_w = v^{-\frac{\ell(w)}2} T_w.
\end{equation}
The element $f^*_{\la}$
is related to the central element $\Ga_{\la}$ used in \cite{Fra, La}
via
$$
f_\la^*=v^{-\frac{n-\ell(\la)}{2}}\Ga_{\la}, \quad \text{ for }\la\in\mc P_n.
$$

%
\subsection{Central idempotents}
Since $\HC$ is semisimple and $\tau$ is a symmetrizing trace function, we can write
\begin{equation}\label{selement}
\tau=\sum_{\la\in\mc P_n}\frac{1}{\kappa_\la(v)}\chi^{\la}_v,
\end{equation}
where $\kappa_\la(v) \in \Z[v,v^{-1}]$ are known as the {\em Schur elements} for $\HC$ (cf. \cite[Theorem ~7.2.6]{GP} where the notation $c_\la$ is used for Schur elements).

Denote by
$$
\ds P_{n}(v)=\sum_{\sigma\in \mf S_n}v^{\ell(\sigma)}=\frac{\prod^n_{k=1}(v^k-1)}{(v-1)^n}
$$ the
Poincar{\'e} polynomial of the symmetric group $\mf S_n$.
The {\em generic degree} $d_\la(v)$ of $\chi^\la_v$ is defined to be
\begin{equation}\label{generic}
d_\la(v)=\frac{P_n(v)}{\kappa_\la(v)}.
\end{equation}
It turns out that $d_\la(v)$ is a polynomial in $v$, and $d_\la(1)$ equals the degree
$d_\la$ of the irreducible character
$\chi^\la$.
A closed formula for $d_\la(v)$ due to
Steinberg can be found in  \cite[Theorem 10.5.3]{GP}.

For $\la\in\mc P_n$, set
$$
\ee^\la_v=\frac{1}{\kappa_\la(v)}\sum_{w\in\mf S_n}\chi^\la_v(T_w)T_w^\vee.
$$

\begin{lemma}\cite[Proposition 4.2]{GR}\label{idempotent}
The following holds for $\la\in\mc P_n$:
\begin{enumerate}
\item $\ee^{\la}_v$ is a central primitive idempotent in $\HC$.

\item $\ds \ee^\la_v=\frac{1}{\kappa_\la(v)}(\chi^\la_v)^*
=\frac{1}{\kappa_\la(v)}\sum_{\mu\in\mc P_n}\chi^\la_v(T_{w_\mu})f^*_\mu$.
\end{enumerate}
\end{lemma}
Since the specializations of $\tau$ and $\chi^\la_v$ at $v=1$
coincide with the canonical trace function on $\C\mf S_n$ and the irreducible character $\chi^\la$,
respectively, $\ee^\la_v$ specializes to the central primitive idempotent
$\ee^\la$ in $\C\mf S_n$ given by
\begin{equation}\label{idempotent0}
\ee^\la=\frac{d_\la}{n!}\sum_{\mu\in\mc P_n}\chi^\la(w_\mu)c_\mu.
\end{equation}

\section{The Frobenius map $\Psi$}
\label{sec:Psi}

\subsection{Lascoux's results}\label{subsec:Z}
For a composition $\alpha=(\al_1,\ldots, \al_{\ell})\in\mc{CP}_n$,
denote by $\mc H_{\al}$ the Hecke algebra associated to the
Young subgroup $\mf S_{\al}$, which can be identified with the subalgebra of $\HC$
generated by the elements $T_i$ with $1\leq i\leq n-1$ and
$i\neq \al_1,\al_1+\al_2,\ldots,\al_{1}+\cdots+\al_{\ell-1}$.
Denote by $\mc D_{\al}$ the set of minimal length coset representatives of $\mf S_{\al}$ in $\mf S_n$.
For $h\in\mc H_{\al}$, define the relative norm \cite{Jo}
\begin{equation}\label{eq:norm}
\mc N_{\al}(h):=\sum_{w\in \mc D_{\al}}v^{-\ell(w)}T_whT_{w^{-1}}.
\end{equation}
The following fundamental result is due to Jones.

\begin{proposition}\cite[Proposition 2.13]{Jo}
$\mc N_{\al}(h)\in\mc Z(\HC)$ if $h$ is central in $\mc H_{\al}$.
\end{proposition}

Denote by $\La$ the ring of symmetric functions in variables $x=\{x_1, x_2,\ldots\}$ over
$\C$  and let $\La^k$ be the subspace of $\La$
consisting of symmetric functions of degree $k$ for $k\geq 0$.
Then $\La^n$ admits several distinguished bases:
the monomial symmetric functions $\{m_{\la}~|\la\in\mc P_n\}$,
the elementary symmetric functions $\{e_{\la}~|\la\in\mc P_n\}$,
the homogeneous symmetric functions $\{h_{\la}~|\la\in\mc P_n\}$,
the power symmetric functions $\{p_{\la}~|\la\in\mc P_n\}$,
and the Schur functions $\{s_{\la}~|\la\in\mc P_n\}$.
Following \cite{Mac}, a transition matrix $M(u,v)$ from
a basis $\{u_\la\}$ to another basis $\{v_\la\}$ means that $u_\la =\sum_\mu M(u,v)_{\la,\mu} v_\mu$.
 Denote by $M(e,m)$ (resp. $M(h,m)$, $M(p,m)$) the transition matrix
from the basis $\{e_{\la}~|\la\in\mc P_n\}$
(resp. $\{h_{\la}~|\la\in\mc P_n\}$, $\{p_{\la}~|\la\in\mc P_n\}$)
to the basis $\{m_{\la}~|\la\in\mc P_n\}$ of $\La^n$.
Let
\begin{equation}\label{eq:D}
D={\rm diag}\{(v-1)^{n-\ell(\la)}\}_{\la\in\mc P_n}
\end{equation}
be the diagonal matrix with entries
$(v-1)^{n-\ell(\la)}$ for $\la\in\mc P_n$.

\begin{proposition}\cite[Theorem 6, Theorem 7]{La}
\label{prop:Lascoux}
\begin{enumerate}
\item Given a composition $\alpha=(\alpha_1,\ldots,\alpha_\ell)\in\mc{CP}_n$, the central element
$\mc N_\alpha(1)$ can be written as
$$
\mc N_\alpha(1)=\sum_{\mu\in\mc P_n}(v-1)^{n-\ell(\mu)}b_{\al\mu} f_\mu^*,
$$
where $b_{\al\mu}$ is equal to the number of the 0-1 matrices having row sums $\al$
and columns sums $\mu$.

\item
The set $\{\mc N_{\la}(1)~|\la\in\mc P_n\}$ is a basis of $\mc Z(\HC)$.
Moreover the transition matrix from $\{\mc N_{\la}(1)~|\la\in\mc P_n\}$ to the $v$-class sum basis
$\{f^*_{\la}~|\la\in\mc P_n\}$ is equal to the product
$$
M(e,m)\cdot D.
$$

\end{enumerate}
\end{proposition}

For a one-part partition $\la=(n)$, we shall simplify notation $\mc N_{(n)}(1)$
by $\mc N_n(1)$, and so on.
Let $w_n^0$ be the longest element in $\mf S_n$.
Recall the notation $\widetilde{T}_w$
from \eqref{eq:Ttilde}.

\begin{remark}
Part (1) of Proposition~\ref{prop:Lascoux} is a special case of
\cite[Theorem 6]{La} and (2) follows immediately from (1).
Another interesting special case of~\cite[Theorem 6]{La} states that
\begin{equation}\label{eq:T2}
\mc N_n(\widetilde{T}^2_{w^0_n})=\sum_{\mu\in\mc P_n}(v-1)^{n-\ell(\mu)}f_\mu^*.
\end{equation}
\end{remark}
\subsection{The algebra $\mc Z$}

Consider the direct sum of the centers
\begin{equation}
\mc Z:=\oplus_{n\geq 0}\mc Z(\HC).
\end{equation}
We define a product $\circ$ on $\mc Z$ via the relative norm~\eqref{eq:norm} as follows:
for $z_1\in\mc Z(\mc H_m)$ and $z_2\in\mc Z(\mc H_n)$, set
$$
z_1\circ z_2=\mc N_{(m,n)}(z_1\otimes z_2),
$$
where $z_1\otimes z_2$ is viewed as an element in the algebra $\mc H_m\otimes\mc H_n\cong\mc H_{(m,n)}$
associated to the composition $(m,n)$ of $m+n$.

\begin{lemma}\cite[Lemma 2.12]{Jo}
The following holds:
\begin{equation} \label{eq:assoc}
\mc N_{(n_1+n_2,n_3)}(\mc N_{(n_1,n_2)}(h_1\otimes h_2)\otimes h_3)
=\mc N_{(n_1,n_2,n_3)}(h_1\otimes h_2\otimes h_3)
\end{equation}
for $n_i\geq 0$ and $h_i\in\mc Z(\mc H_{n_i})$ with $1\leq i\leq 3$.
\end{lemma}

\begin{theorem}\label{thm:algebraZ}
The algebra $(\mc Z,\circ)$ is commutative, associative, and is
isomorphic to a polynomial algebra with generators $\mc N_n(1),$ for $n\geq 1$.
\end{theorem}
\begin{proof}
By~\eqref{eq:assoc}, the product $\circ$ on $\mc Z$ is associative,  and
moreover
$$
\mc N_\la(1)=\mc N_{\la_1}(1)\circ\cdots\circ\mc N_{\la_\ell}(1)
$$
for partition $\la=(\la_1,\ldots,\la_\ell)$.
Then by Proposition~\ref{prop:Lascoux}(2) we deduce that
the algebra $\mc Z$ is generated by $\mc N_n(1), n\geq 1$.
Meanwhile by Proposition~\ref{prop:Lascoux}(1), we have
$\mc N_{(m,n)}(1)=\mc N_{(n,m)}(1)$
and hence $\mc N_m(1)\circ\mc N_n(1)=\mc N_n(1)\circ\mc N_m(1)$
for $m,n\geq 1$.
Thus $\mc Z$ is commutative.

Furthermore by Proposition~\ref{prop:Lascoux}(2), a basis of $\mc Z$
is given by the monomials in $\mc N_n(1)$ for $n\geq 1$, and
hence $\mc Z$ is a polynomial algebra generated by $\mc N_n(1),$ for $n\geq 1$.
\end{proof}

\subsection{The symmetric functions $\undm_\la$}

Recall that $\La$ denotes the ring of symmetric functions in variables
$x=\{x_1,x_2,\ldots\}$ over $\C$.
In $\La_v:=\C(v)\otimes_{\C}\La$, we define for $k\ge 1$ that
$$
p_k((v-1)x):=(v^k-1)p_k(x),
\qquad
p_k\Big(\frac{x}{v-1}\Big) :=\frac{p_k(x)}{v^k-1}.
$$
Then the $\la$-ring notation
$f((v-1)x)$ and $f(\frac{x}{v-1})$ are defined in $\La_v$ accordingly
for any symmetric function $f(x)$ by first expressing $f$ in terms of
power-sums.
For $\la\in\mc P_n$, set
\begin{equation} \label{eq:mbar}
\undm_{\la}(x) :=(v-1)^{\ell(\la)}m_{\la}\Big(\frac{x}{v-1}\Big).
\end{equation}

\begin{example}
Since $\ds m_{(1,1)}=\frac{1}{2}p_{(1,1)}-\frac{1}{2}p_2$
and $m_2=p_2$, we obtain
$\undm_{\mu}$ for $\mu=(1,1)$ and $\mu=(2)$
as follows:
\begin{align*}
\undm_{(1,1)}(x)
&= \frac{1}{2}(v-1)^2\Big(p_{(1,1)}\Big(\frac{x}{v-1}\Big)
-p_2\Big(\frac{x}{v-1}\Big)\Big)\\
&=\frac{1}{2}p_1(x)^2-\frac{v-1}{2(v+1)}p_2(x),  \\
\undm_{(2)}(x)
&=(v-1)p_2\Big(\frac{x}{v-1}\Big)
=\frac{1}{v+1}p_2(x).
\end{align*}
\end{example}
\subsection{The Frobenius map $\Psi$}

\begin{definition}
Define the (quantum) Frobenius  map to be the linear map
\begin{align}
\Psi: \mc Z &\longrightarrow \La_v,  \notag\\
f^*_{\la}&\mapsto
\undm_{\la}(x)
\label{Psi}, \quad
\text{ for all partitions } \la.
\end{align}
\end{definition}
%

\begin{proposition}\label{prop:N1}
The following holds for $\la\in\mc P_n$:
$$
\Psi(\mc N_{\la}(1))
=(v-1)^n e_{\la}\Big(\frac{x}{v-1}\Big).
$$
\end{proposition}
\begin{proof}
In the notation of Proposition~\ref{prop:Lascoux}, we write $e_{\la}(x)=\sum_{\mu}b_{\la\mu}m_{\mu}(x)$.
By Proposition~\ref{prop:Lascoux}(2) we obtain
$$
\mc N_{\la}(1)=\sum_{\mu\in\mc P_n}b_{\la\mu}(v-1)^{n-\ell(\mu)}f^*_{\mu}.
$$
We compute that
\begin{align*}
\Psi(\mc N_{\la}(1))
&=\sum_{\mu\in\mc P_n}b_{\la\mu}(v-1)^{n-\ell(\mu)}\Psi(f^*_{\mu})\\
&=\sum_{\mu\in\mc P_n}b_{\la\mu}(v-1)^{n-\ell(\mu)}
\undm_{\mu}(x)\\
&=(v-1)^n\sum_{\mu\in\mc P_n}b_{\la\mu}m_{\mu}\Big(\frac{x}{v-1}\Big)\\
&=(v-1)^n e_{\la}\Big(\frac{x}{v-1}\Big).
\end{align*}
The proposition is proved.
\end{proof}

\begin{theorem}\label{thm:Psi}
The Frobenius map $\Psi:\mc Z\longrightarrow \La_v$ is an algebra isomorphism.
\end{theorem}
\begin{proof}
Clearly $\Psi$ is a linear isomorphism
since $\{f_\la^*~|\la\in\mc P_n\}$ and $\{\undm_\la(x)~|\la\in\mc P_n\}$
are $\C(v)$-linear basis of $\mc Z(\HC)$ and $\La^n_v$, respectively, for $n\geq 0$.

It remains to show that $\Psi$ is
an algebra homomorphism.
Suppose $\mu=(\mu_1,\ldots,\mu_k)\in\mc P_m$ and
$\la=(\la_1,\ldots,\la_\ell)\in\mc P_n$.
Then by~\eqref{eq:assoc} we have
\begin{align*}
\mc N_\mu(1)\circ\mc N_\la(1)
&=\mc N_{(m,n)}(\mc N_\mu(1)\otimes\mc N_\la(1))\\
&=\mc N_{\mu_1}(1)\circ\cdots\circ\mc N_{\mu_k}(1)
\circ\mc N_{\la_1}(1)\circ\cdots\circ\mc N_{\la_\ell}(1)\\
&=\mc N_{\mu\cup\la}(1),
\end{align*}
where $\mu\cup\la$ is the partition corresponding
to the composition obtained by concatenating the parts of $\mu$ and $\la$,
and the last equality is due to the commutativity of $\mc Z$ by Theorem~\ref{thm:algebraZ}.
Then by Proposition~\ref{prop:N1} one deduces that
\begin{align*}
\Psi(\mc N_\mu(1)\circ\mc N_\la(1))
&=(v-1)^{m+n}e_{\mu\cup\la}\Big(\frac{x}{v-1}\Big)\\
&=(v-1)^m e_\mu\Big(\frac{x}{v-1}\Big)
\cdot (v-1)^ne_\la\Big(\frac{x}{v-1}\Big)\\
&=\Psi(\mc N_\mu(1))\Psi(\mc N_\la(1)).
\end{align*}
This proves the theorem.
\end{proof}

\begin{remark}
We mention that there is another connection in a different direction
between the center  $\mathcal Z(\HC)$ and
the ring of symmetric functions, which has yet to be developed \cite{FW}.
In the $v=1$ case, such a connection was developed in \cite[I.
7, Ex.~25]{Mac}, which
was in turn based on
a classical work of Farahat-Higman.
\end{remark}

\section{Identification of bases of $\mc Z$ and $\La_v$}
\label{sec:identification}
\subsection{Identification of bases}
Given a partition $\la=(\la_1,\ldots, \la_{\ell})$,
let $w^0_{\la}$ be the longest element in the Young subgroup
$\mf S_{\la}$,
and set
$$
F_{\la}:=f^*_{\la_1}\otimes f^*_{\la_2}\otimes\cdots\otimes f^*_{\la_{\ell}}
\in\mc H_{\la},
$$
where $f^*_{\la_i}$ is the Geck-Rouquier class element in the $i$th factor $\mc H_{\la_i}$ of $\mc H_\la$
associated to the one-part partition $(\la_i)$ for $1\leq i\leq\ell$.
The element $F_{\la}$ is related to the element
$e_{\la^\natural}=\Ga_{\la_1}\otimes\Ga_{\la_2}\ldots\otimes\Ga_{\la_{\ell}}$
in \cite{La} for $\la=(\la_1,\ldots,\la_\ell)$  via
$F_{\la}=v^{-\frac{n-\ell(\la)}{2}}e_{\la^\natural}.$

\begin{proposition}
\label{prop:NFT2} Suppose $\la$ is a partition of $n$. Then
\begin{enumerate}
\item
$\ds \Psi(\mc N_{\la}\big(\widetilde{T}^2_{w^0_{\la}})\big) =(v-1)^n
h_{\la}\Big(\frac{x}{v-1}\Big),$

\item
$
\ds \Psi(\mc N_{\la}(F_{\la}))
=(v-1)^{\ell(\la)}p_{\la}\Big(\frac{x}{v-1}\Big).
$
\end{enumerate}
\end{proposition}
\begin{proof}
Assume $\la=(\la_1,\la_2,\ldots,\la_\ell)$
with $\ell(\la)=\ell$.
By~\eqref{eq:T2}, one deduces that, for $n\geq 1$,
\begin{align}
\Psi(\mc N_{n}\big(\widetilde{T}^2_{w^0_n})\big)
&=\sum_{\mu\in\mc P_n}(v-1)^{n-\ell(\mu)}\Psi(f^*_\mu)\notag\\
&=\sum_{\mu\in\mc P_n}(v-1)^n m_\mu\Big(\frac{x}{v-1}\Big)\notag\\
&=(v-1)^nh_n\Big(\frac{x}{v-1}\Big), \label{eq:NT21}
\end{align}
where the last equality is due to fact that $h_n=\sum_{\mu\in\mc P_n}m_\mu$.
Observe that $\widetilde{T}^2_{w^0_{\la}}$ can be written as
$
\widetilde{T}^2_{w^0_{\la}}=\widetilde{T}^2_{w^0_{\la_1}}\otimes\cdots\otimes
\widetilde{T}^2_{w^0_{\la_\ell}}
$
and hence
$$
\mc N_\la(\widetilde{T}^2_{w^0_{\la}})=\widetilde{T}^2_{w^0_{\la_1}}\circ\cdots\circ
\widetilde{T}^2_{w^0_{\la_\ell}}.
$$
Therefore, by Theorem~\ref{thm:Psi} and~\eqref{eq:NT21}, we calculate
\begin{align*}
\Psi(\mc N_\la(\widetilde{T}^2_{w^0_{\la}})))
&=\Psi(\widetilde{T}^2_{w^0_{\la_1}})\cdots
\Psi(\widetilde{T}^2_{w^0_{\la_\ell}})\\
&=(v-1)^{\la_1}h_{\la_1}\Big(\frac{x}{v-1}\Big)\cdots
(v-1)^{\la_\ell}h_{\la_\ell}\Big(\frac{x}{v-1}\Big)\\
&=(v-1)^nh_\la\Big(\frac{x}{v-1}\Big).
\end{align*}
This proves Part (1) of the proposition.

By definition, we have
$$\mc N_\la(F_\la)=f^*_{\la_1}\circ\cdots\circ f^*_{\la_\ell}.
$$
Note that $m_k=p_k$ for $k\ge 1$.
Then by Theorem~\ref{thm:Psi}, we compute that
\begin{align*}
\Psi(\mc N_\la(F_\la))
&=\Psi(f^*_{\la_1})\cdots\Psi(f^*_{\la_\ell})\\
&=(v-1)m_{\la_1}\Big(\frac{x}{v-1}\Big)\cdots(v-1)m_{\la_\ell}\Big(\frac{x}{v-1}\Big)\\
&=(v-1)^{\ell(\la)}p_\la\Big(\frac{x}{v-1}\Big).
\end{align*}
Hence Part (2) of the proposition is proved.
\end{proof}
\begin{remark}
$\widetilde{T}^2_{w^0_n}$ is the product of
all ``multiplicative'' Jucys-Murphy elements in $\HC$ (cf. \cite{La}).
Meanwhile  $F_n=f_n^*\in\mc Z(\HC)$
is the product of all  ``additive''   Jucys-Murphy
elements in $\HC$~\cite{Jo} (also see \cite{La}).
\end{remark}
Recall the matrix $D$ from~\eqref{eq:D}.

\begin{corollary}\label{cor:NF}
The following holds for $n\geq 0$:
\begin{enumerate}

\item The set $\{\mc N_{\la}(\widetilde{T}^2_{w^0_{\la}})~|\la\in\mc P_n\}$
is a basis of $\mc Z(\HC)$ and the transition matrix from this basis to
$\{f^*_{\la}~|\la\in\mc P_n\}$ is equal to the product
$$
M(h,m)\cdot D.
$$

\item The set $\{\mc N_{\la}(F_{\la})~|\la\in\mc P_n\}$ is a basis of
$\mc Z(\HC)$ and the transition matrix from this basis to
$\{f^*_{\la}~|\la\in\mc P_n\}$ is equal to the product
$$
D^{-1}\cdot M(p,m)\cdot D.
$$
\end{enumerate}
\end{corollary}
\begin{proof}
Observe that  $\ds \Big\{(v-1)^nh_\la\Big(\frac{x}{v-1}\Big)~|\la\in\mc P_n\Big\}$ is a basis of $\La^n_v$.
Then by Proposition~\ref{prop:NFT2}(1), the set $\big\{\mc N_{\la}(\widetilde{T}^2_{w^0_{\la}})~|\la\in\mc P_n\big\}$
is a basis of $\mc Z(\HC)$. Part (1)
follows now by combining~\eqref{Psi} and Proposition~\ref{prop:NFT2}(1).
By a similar argument, Part (2) can be verified via Proposition~\ref{prop:NFT2}(2).
\end{proof}
\begin{remark}
Corollary~\ref{cor:NF} was originally proved by Lascoux~\cite[Theorem~ 10, Theorem~ 12]{La},
while the basis $\{\mc N_\la(F_\la)~|\la\in\mc P_n\}$ was constructed earlier by Jones~\cite{Jo}.
Lascoux's proof used ingenious connection to non-commutative symmetric functions
which is difficult for us to follow.
In his approach, Lascoux established the transition matrices in  the three cases as in Proposition~\ref{prop:Lascoux}
and Corollary~\ref{cor:NF} separately.
\end{remark}

\begin{corollary}
\begin{enumerate}
\item $\mc Z$ is a polynomial algebra with generators $\mc N_n(\widetilde{T}^2_{w^0_n}),$
for $n\geq 1$.

\item $\mc Z$ is a polynomial algebra with generators $\mc N_n(F_n),$ for  $n\geq 1$.
\end{enumerate}
\end{corollary}
In the following, we summarize the identification of
various bases of $\mc Z$ and $\La_v$, for the convenience of the reader.

\begin{theorem}\label{PsiLascoux}
Suppose $\la\in\mc P_n$.
Then the following holds:
\begin{align}
\Psi(f^*_{\la}) &=(v-1)^{\ell(\la)}
m_{\la}\Big(\frac{x}{v-1}\Big),
 \label{eq:f}  \\
\Psi(\ee^\la_v) &=\frac{1}{\kappa_\la(v)}s_\la(x),\label{eq:e}\\
\Psi(\mc N_{\la}(1))
&=(v-1)^n e_{\la}\Big(\frac{x}{v-1}\Big),
  \label{eq:N1} \\
\Psi(\mc N_{\la}(\widetilde{T}^2_{w^0_{\la}}))
&=(v-1)^n h_{\la}\Big(\frac{x}{v-1}\Big),
  \label{eq:NT2} \\
\Psi(\mc N_{\la}(F_{\la}))
&=(v-1)^{\ell(\la)}p_{\la}\Big(\frac{x}{v-1}\Big).
 \label{eq:NF}
\end{align}
\end{theorem}

\begin{proof}
Note that \eqref{eq:f} follows from the definition of $\Psi$,
\eqref{eq:N1}, \eqref{eq:NT2} and \eqref{eq:NF} follow from
Proposition~\ref{prop:N1} and Proposition~\ref{prop:NFT2}.
One the other hand, \eqref{eq:e} follows
from the two identities~\eqref{v-ch} and~\eqref{eq:v-ch1}
in the subsequent section (This can be viewed as a reformulation of
the main result of~\cite{Ram}).
\end{proof}

\begin{remark}
It follows by~\eqref{eq:e} that the restriction $\Psi_n:\mc Z(\HC)\rightarrow\La^n_v$ of $\Psi$
is an algebra isomorphism if we impose an algebra structure on $\La^n_v$
by $s_\la\ast s_\mu=\delta_{\la\mu}\kappa_\la(v)s_\la$.
\end{remark}
%

\subsection{Specialization at $v=1$}
For a partition $\la=(1^{m_1}2^{m_2}\cdots)$, define
$$
z_{\la}=\prod_{i\geq 1}i^{m_i} m_i!.
$$

It is known that the Frobenius morphism from $\C \mf S_n$ to $\La^n $
sends $w$ to $p_\la(x)$ if $w\in C_{\la}$.
This gives rise to an isomorphism (called the classical Frobenius map):
\begin{align}\label{Frobenius}
\psi:\oplus_{n\geq 0}\mc Z(\C\mf S_n)&\longrightarrow \La,  \\
c_{\la}&\mapsto \frac{1}{z_{\la}}p_{\la}(x), \quad
\text{ for } \la\in\mc P_n.
\notag
\end{align}

For a symmetric function $f\in \La_v$, denote when possible by
$f|_{v=1}$ the specialization of $f$ at $v=1$.
\begin{lemma}\label{mbar}
The following holds for any partition $\la$:
$$
\undm_{\la}(x)|_{v=1}
=\frac{1}{z_{\la}}p_{\la}(x).
$$
\end{lemma}
\begin{proof}
Let us write $\la$ as $\la =(1^{m_1}2^{m_2}\ldots)$.
By inverting the transition from $p_\mu$ to $m_\la$ \cite[I, (6.8), (6.9)]{Mac},
we obtain that
$$
m_{\la}(x)=\sum_{\mu}a_{\la\mu}p_{\mu}(x),
$$
where
\begin{align}
a_{\la\la}&=\prod_{i\geq 1}\frac{1}{m_i!},
  \notag \\
a_{\la\mu}\neq 0  \ (\la \ne \mu) & \Longrightarrow \mu\geq\la \text{ and } \ell(\mu)< \ell(\la).
\label{ptom}
\end{align}
This implies that
\begin{align*}
\undm_{\la}(x)
&=(v-1)^{\ell(\la)}
\sum_{\mu}a_{\la\mu}p_{\mu}\big(\frac{x}{v-1}\big)\\
&=(v-1)^{\ell(\la)}
\sum_{\mu}a_{\la\mu}\prod_{k=1}^{\ell(\mu)}\frac{1}{v^{\mu_k}-1}p_{\mu}(x)\\
&=\sum_{\mu}a_{\la\mu}
\frac{(v-1)^{\ell(\la)}}{\prod_{k=1}^{\ell(\mu)}(v^{\mu_k}-1)}p_{\mu}(x).
\end{align*}
Let $\mu\in\mc P_n$ be such that $a_{\la\mu}\neq 0$.
Then \eqref{ptom} holds, and
$$
\frac{(v-1)^{\ell(\la)}}{\prod_{k=1}^{\ell(\mu)}(v^{\mu_k}-1)}|_{v=1}
=\left\{
\begin{array}{cc}
0,&\text{ if }\ell(\mu)<\ell(\la),\\
\frac{1}{\prod^{\ell(\mu)}_{k=1}\mu_k},&\text{ if }\ell(\mu)=\ell(\la).
\end{array}
\right.
$$
Now the condition $\ell(\mu)=\ell(\la)$ implies that $\mu=\la$
by \eqref{ptom}, and hence
\begin{align*}
\undm_{\la}(x)|_{v=1}
&=a_{\la\la}\frac{1}{\prod^{\ell(\la)}_{k=1}\la_k}p_{\la}(x)\\
&=\frac{1}{\prod_{i\geq 1} m_i!i^{m_i}}p_{\la}(x)
=\frac{1}{z_{\la}}p_{\la}(x).
\end{align*}
The lemma is proved.
\end{proof}
By~\eqref{Frobenius}, Lemma~\ref{mbar} and the fact that $f^*_{\la}$ specializes to $c_{\la}$ at $v=1$,
we have the following.
\begin{proposition}\label{prop:Psin}
The specialization of the (quantum) Frobenius map
$\Psi$ at $v=1$ coincides with the classical Frobenius map $\psi$.
\end{proposition}

\section{The $v$-characteristic map} \label{sec:v-ch}
\subsection{The $v$-characteristic map ${\rm ch}_v$}
Denote by $R^n$ and $R^n_v$ the Grothendieck groups of the categories of
$\C\mf S_n$-modules and $\HC$-modules, respectively.
Set
$$
R=\bigoplus_{n\geq 0}\C\otimes_{\Z}R^n,\quad
R_v=\bigoplus_{n\geq 0}\C(v)\otimes_{\Z}R^n_v.
$$
It is known that $R$ is a commutative algebra with
the multiplication given by
$$
[M]\cdot[N]= \big[{\rm ind}^{\mf S_{m+n}}_{\mf S_m\otimes\mf S_n}M\otimes N \big]
$$
for $\mf S_m$-module $M$ and $\mf S_n$-module $N$.
Similarly, $R_v$ also affords an algebra structure via the functor
${\rm ind}^{\mc H_{m+n}}_{\mc H_m\otimes\mc H_n}$.
Below we freely identify elements in $\C\otimes_{\Z}R^n$ and $\C(v)\otimes_{\Z}R^n_v$
with trace functions on $\C\mf S_n$ and $\HC$, respectively.

The Frobenius characteristic map (cf. \cite[I. 7]{Mac})
\begin{align*}
{\rm ch}: R&\longrightarrow \La\\
\chi&\mapsto
\sum_{\mu\in\mc P_n}\chi(w_{\mu})\psi(c_{\mu}) \quad \text{ for }\chi\in R^n
\end{align*}
sends the irreducible character
$\chi^{\la}$ to the Schur function $s_{\la}$ for any partition $\la$.
Observe that the characteristic map ${\rm ch}$ can also be described using
$\ee^\la$ in \eqref{idempotent0} for $\la\in\mc P_n$ as
\begin{equation}  \label{eq:chpsi}
{\rm ch}(\chi^\la)=\frac{n!}{d_\la}\psi(\ee^\la).
\end{equation}

Recall the generic degrees $d_\la(v)$ from \eqref{generic}.
In light of \eqref{eq:chpsi},
we make the following.

\begin{definition}
Define the Frobenius characteristic map  ${\rm ch}_v$ to be the linear map by
\begin{align}
{\rm ch}_v: R_v&\longrightarrow \La_v,
  \label{v-ch}\\
\chi^\la_v&\mapsto \frac{P_n(v)}{d_\la(v)}\Psi(\ee^\la_v), \quad \forall \la \in \mc P_n.
 \notag
\end{align}
\end{definition}
%

For $r\geq 1$ and for partition $\mu=(\mu_1,\ldots,\mu_{\ell})$ of length $\ell$,
we set
\begin{align}
\overline{h}_r(x)&=\frac{1}{v-1}h_r((v-1)x) \notag,%
  \\
\overline{h}_{\mu}(x)&=\overline{h}_{\mu_1}(x)\overline{h}_{\mu_2}(x)
\ldots \overline{h}_{\mu_{\ell}}(x)
=\frac{1}{(v-1)^{\ell(\mu)}}h_\mu((v-1)x).
 \label{eq:qbar}
\end{align}

Recall that $\{\chi^{\la}_v~|\la\in\mc P_n\}$ is a complete set of
irreducible characters of the Hecke algebra $\HC$, and
recall the following  Frobenius character formula for $\HC$ of Ram.

\begin{proposition}\cite[Theorem 4.14]{Ram}
Suppose $\mu\in\mc P_n$.
Then
\begin{equation}\label{Ram}
\overline{h}_{\mu}(x)=\sum_{\la\in\mc P_n}\chi^{\la}_v(T_{w_{\mu}})s_{\la}(x).
\end{equation}
\end{proposition}

The Frobenius character formula admits the following dual reformulation.

\begin{proposition}
Suppose $\la\in\mc P_n$. Then
\begin{align}
s_{\la}(x)
&=\sum_{\mu\in\mc P_n}\chi^{\la}_v(T_{w_{\mu}})\undm_{\mu}(x).
\label{Ramdual}
\end{align}
\end{proposition}

\begin{proof}
We have the following  (cf. \cite[III, (4.2), (4.7)]{Mac})
\begin{align*}
\prod_{i,j}\frac{1-x_iy_j}{1-vx_iy_j}
&=\sum_{\la}h_{\la}((v-1)x)m_{\la}(y).
\end{align*}
Hence we have by \eqref{eq:mbar} and \eqref{eq:qbar} that
\begin{align*}
\prod_{i,j}\frac{1}{1-x_iy_j}
&=\sum_{\la}h_{\la}((v-1)x)m_{\la}\Big(\frac{y}{v-1}\Big)
 \notag\\
&=\sum_{\la}  \overline{h}_{\la}(x) \undm_{\la}(y).
\end{align*}
This together with the Cauchy identity
$
\prod_{i,j}\frac{1}{1-x_iy_j}
=\sum_{\la}s_{\la}(x)s_{\la}(y)
$
implies that
\begin{align}
\sum_{\la}  \overline{h}_{\la}(x) \undm_{\la}(y)=\sum_{\la}s_{\la}(x)s_{\la}(y).
\label{Cauchy}
\end{align}
This last identity \eqref{Cauchy} can be reinterpreted by saying that
$ \overline{h}_{\mu}(x)$  and $\undm_{\mu}(y)$ form
dual bases with respect to a bilinear form in $\La_v$
such that $s_\la$ are orthonormal.
Hence \eqref{Ramdual} follows from, and is indeed equivalent to,
\eqref{Ram}.
\end{proof}

\begin{theorem}\label{thm:v-ch}
The following holds for all partitions $\la$:
\begin{equation}\label{eq:v-ch1}
{\rm ch}_v(\chi^{\la}_v)=s_{\la}(x).
\end{equation}
\end{theorem}
\begin{proof}
Recall from~\eqref{selement} that $\kappa_\la(v)$ are the Schur elements.
By~\eqref{generic}, \eqref{v-ch} and Lemma~\ref{idempotent}, we obtain
\begin{align*}
{\rm ch}_v(\chi^{\la}_v)
&=\frac{P_n(v)}{d_\la(v)\kappa_\la(v) }\sum_{\mu\in\mc P_n}\chi^{\la}_v(T_{w_{\mu}})\Psi(f^*_\mu)\\
&=\sum_{\mu\in\mc P_n}\chi^{\la}_v(T_{w_{\mu}})\undm_{\mu}(x).
\end{align*}
The theorem now follows by \eqref{Ramdual}.
\end{proof}

\begin{remark}
Fomin and Thibon \cite{FT} took the correspondence in Theorem~\ref{thm:v-ch}
between irreducible characters and Schur functions as a definition
of the characteristic map for Hecke algebras,
though they did not study the centers.
\end{remark}
\subsection{A commutative diagram}
\begin{lemma}\label{lem:algebraRv}
Suppose $\mu\in\mc P_m,\la\in\mc P_n$.
Then
$$
{\rm ind}^{\mc H_{m+n}}_{\mc H_m\otimes\mc H_n}(\chi^\mu_v\otimes\chi^\la_v)
=\sum_{\nu\in\mc P_{m+n}}c^\nu_{\mu\la}\chi^\nu_v,
$$
where $c^\nu_{\mu\la}$ are the Littlewood-Richardson coefficients.
\end{lemma}
\begin{proof}
Assume $\mu\in\mc P_m,\la\in\mc P_n$.
The induced character
${\rm ind}^{\mc H_{m+n}}_{\mc H_m\otimes\mc H_n}(\chi^\mu_v\otimes\chi^\la_v)$
can be written as
$$
{\rm ind}^{\mc H_{m+n}}_{\mc H_m\otimes\mc H_n}(\chi^\mu_v\otimes\chi^\la_v)
=\sum_{\nu\in\mc P_{m+n}}u_\nu\chi^\nu_v,
$$
for some non-negative integer $u_\nu$.
Then specializing at $v=1$, we obtain
$$
{\rm ind}^{\mf S_{m+n}}_{\mf S_m\otimes\mf S_n}(\chi^\mu\otimes\chi^\la)
=\sum_{\nu\in\mc P_{m+n}}u_\nu\chi^\nu.
$$
Then it is known from the representation theory of symmetric groups that
$u_\nu$ coincides with the Littlewood-Richardson coefficients $c^\nu_{\mu\la}$
for $\nu\in\mc P_{m+n}$.
\end{proof}

Then by~\eqref{eq:v-ch1} and Lemma \ref{lem:algebraRv}, we have the following.

\begin{theorem}
The characteristic map ${\rm ch}_v: R_v\rightarrow\La_v$ is an algebra isomorphism.
\end{theorem}

Recall from~\eqref{dual} that the map
\begin{align*}
\Theta_n:\C(v)\otimes_{\Z} R^n_v&\longrightarrow \mc Z(\HC)\\
\chi&\mapsto\chi^*=\sum_{w\in\mf S_n}\chi(T_w)T_w^\vee
\end{align*}
is an isomorphism.
This leads to a linear isomorphism
$$
\Theta=\oplus_{n\geq 0}\Theta_n: R_v\longrightarrow \mc Z.
$$

\begin{proposition}
We have a commutative diagram of isomorphisms of graded algebras:
\begin{equation}\label{eq:cd}
\xymatrix{
R_v \ar[r]^{{\Theta}} \ar[d]_{{\rm ch}_v} &\mc Z \ar[ld]^{\Psi}\\
\La_v}
\end{equation}
\end{proposition}
\begin{proof}
It is easy to check that
${\rm ch}_v(\chi^\la_v)=s_\la(x)=\Psi((\chi^\la_v)^*)=\Psi(\Theta(\chi^\la_v))$
for a partition $\la$.
Hence the diagram is commutative.
Since ${\rm ch}_v$ and $\Psi$ are isomorphisms of graded algebra, so is $\Theta$.
\end{proof}

\end{document}